\documentclass[reqno,11pt]{amsart}
\usepackage{latexsym}
\usepackage{amsfonts}
\usepackage{amssymb}
\usepackage{mathrsfs}
\usepackage[all]{xy}
\usepackage{bbm}
\usepackage{setspace}
\usepackage{rotating}
\usepackage{graphicx}
\usepackage{adjustbox}
\usepackage{afterpage}
\usepackage{booktabs, threeparttable}
\usepackage{mathtools, nccmath}
\allowdisplaybreaks

\newtheorem{theorem}{Theorem}[section]

\newtheorem{lemma}[theorem]{Lemma}

\theoremstyle{definition}
\newtheorem{definition}[theorem]{Definition}
\newtheorem{remark}[theorem]{Remark}
\numberwithin{equation}{section}
\DeclareMathOperator{\GL}{GL} 
 \DeclareMathOperator{\SO}{SO}
\DeclareMathOperator{\Aut}{Aut}
\DeclareMathOperator{\ad}{ad}

 \DeclareMathOperator{\SL}{SL}
 \DeclareMathOperator{\End}{End}
\newcommand{\RR}{\mathbb{R}}

\newcommand{\QQ}{\mathbb{Q}}
\newcommand{\ZZ}{\mathbb{Z}}

\newcommand{\sor}{\mathfrak{so}(2,1)\ltimes \mathbb{R}^{3}}
\newcommand{\so}{\mathfrak{so}(2,1)}
\newcommand{\slr}{\mathfrak{sl}(3,\mathbb{R})\ltimes \mathbb{R}^{3}}
\newcommand{\sL}{\mathfrak{sl}(3,\mathbb{R})}
\newcommand{\f}{\mathfrak{f}}
\newcommand{\g}{\mathfrak{g}}
\newcommand{\h}{\mathfrak{h}}
\newcommand{\ke}{\ker(\ad(\underline{u}))}

\title[Oppenheim conjecture for systems of forms]{On the density at integer points of a system comprising an inhomogeneous quadratic form and a linear form}
\author{Prasuna Bandi}
\author{Anish Ghosh}
\address{School of Mathematics, Tata Institute of Fundamental Research, Homi Bhabha Road, Navy Nagar, Colaba, Mumbai 400005, India.}\email{prasuna@math.tifr.res.in, ghosh@math.tifr.res.in}
\thanks{The second named author gratefully acknowledges support from a grant from the Indo-French Centre for the Promotion of Advanced Research; a Department of Science and Technology, Government of India Swarnajayanti fellowship and a MATRICS grant from the Science and Engineering Research Board.}

\begin{document}

\begin{abstract}
We prove an analogue of the Oppenheim conjecture for a system comprising an inhomogeneous quadratic form and a linear form in $3$ variables using dynamics on the space of affine lattices.
\end{abstract}

\maketitle

\section{Introduction}
In this paper, we study the values taken at integer points for a pair consisting of an inhomogeneous quadratic form and a linear form in $3$ variables. 
Let $Q$ be a nondegenerate indefinite quadratic form on $\RR^n$. We say that $Q$ is irrational if $Q$ is not proportional to a quadratic form with integer coefficients. It is a famous theorem  of Margulis \cite{Ma89} resolving an old conjecture of Oppenheim that for an irrational, indefinite, nondegenerate quadratic form $Q$ in $n \geq 3$ variables, $Q(\ZZ^n)$ is dense in $\RR$. We refer to \cite{Borel95} for a nice introduction to the problem and Margulis' proof which involves dynamics on homogeneous spaces, and to \cite{Ma97} for a survey. Subsequently, there have been rapid developments in this subject, quantitative versions were proved in \cite{DM93, EMM1, EMM2} and recently effective versions have been established in \cite{LM, Bo16, GGN, GK}. Inhomogeneous quadratic forms have been studied in \cite{MaMo, Marklof}.  

\subsection*{Inhomogeneous quadratic forms}
Let $Q'$ be an inhomogeneous quadratic form on $\mathbb{R}^{n}$, i.e. $Q'$ is a degree two polynomial in $n$ variables. Then $Q'$ can be written as 
$$Q'(x)=Q(x)+L(x)+c \; \forall \; x\in \RR^{n}$$

\noindent where $Q$ is a homogeneous quadratic form on $\mathbb{R}^{n}$, $L$ is a linear form on $\RR^{n}$ and $c\in \RR$. We say that $Q'$ is indefinite and nondegenerate if $Q$ is indefinite and nondegenerate respectively. The form $Q'$ is said to be irrational if either $Q$ is irrational as a homogeneous quadratic form or $L$ is irrational, i.e. not a scalar multiple of a form with integer coefficients.
A particular kind of inhomogeneous quadratic form is defined as follows. Let $Q$ be a nondegenerate homogeneous quadratic form on $\RR^{n}$ and $\xi \in \RR^n$. Define the inhomogeneous quadratic form $Q_\xi$ by
\begin{equation}\label{def:inho}
Q_{\xi}(x) = Q(x+\xi) \text{ for } x \in \RR^n.
\end{equation}
 It is easy to see that $Q_{\xi}$ is irrational iff either $Q$ is irrational as a homogeneous quadratic form or $\xi$ is an irrational vector, i.e. not a scalar multiple of a vector with integer coordinates.\\



In \cite{MaMo}, Margulis and Mohammadi proved quantitative forms of Oppenheim's conjecture for inhomogeneous forms. Their work contains the qualitative density as a special case. In particular, Theorem $1.4$ in \cite{MaMo} implies that for an indefinite, irrational, nondegenerate inhomogeneous quadratic form $Q_{\xi}$ in $n \geq 3$ variables, $Q_{\xi}(\ZZ^n)$ is dense in $\RR$.

\subsection*{Systems of forms}
The problem of density at integer values for systems of forms dates back to Dani and Margulis \cite{DM90}. They proved that for a  $3$ variable quadratic form $Q$ and a linear form $L$, $$\{(Q(x),L(x))~:~x \in \ZZ^3 \}$$ is dense in $\RR^{2}$ if no nonzero linear combination of $Q$ and $L^2$ is rational, and the plane $\{L = 0\}$ is tangent to the surface $\{Q = 0\}$.  In \cite{Dani00}, Dani proved that if the surface $\{Q = 0\}$ and the plane $\{L = 0\}$ intersect transversally, the density can fail for a set pairs of full Hausdorff dimension. The work of Dani and Margulis was generalised by Gorodnik \cite{Goro04a} who studied pairs comprising a quadratic and linear form in dimensions greater than $3$. Subsequently, he  studied systems of quadratic forms in \cite{Goro04b}. Further progress on systems comprising a quadratic and linear form was made in \cite{Dani08} by Dani. In a related direction, Lazar \cite{Lazar} studied the density of a pair comprising a quadratic and linear form at $S$-integer points, see also the recent paper \cite{Lazar19}. Sargent \cite{Sargent}, studied the density of linear forms at integer points on a quadratic surface.   

\subsection*{Results}
It is a natural question to investigate the density at integer values of systems consisting of inhomogeneous forms. We take the first step in this paper by investigating a pair consisting an inhomogeneous form and a linear form. Our main theorem is:

\begin{theorem}\label{thm:main}
	Let $Q_{\xi}$ be an inhomogeneous, nondegenerate and indefinite quadratic form in $3$ variables and let $L$ be a linear form on $\RR^{3}$. Suppose that:
	\begin{enumerate} 
	\item the plane $\{x\in\RR^{3}\;|\; L(x)=0\}$ is tangential to the cone $\{x\in\RR^{3}\;|\; Q(x)=0\}$ and 
	\item any non-zero linear combination of $Q_{\xi}$ and $L^{2}$ is an irrational quadratic form.
	\end{enumerate}
	Then $\{(Q_{\xi}(x),L(x))\;|\:x\in \ZZ^{3}\}$ is a dense subset of $\RR^{2}$.
\end{theorem}

\noindent \textbf{Remarks}:
\begin{enumerate}
\item Our proof uses the strategy of Margulis, currently the only available strategy for density problems involving forms in low variables, and involves dynamics of group actions on the space of affine lattices in $\RR^{3}$. Condition $(1)$ in Theorem \ref{thm:main} implies that the joint stabilizer of the inhomogeneous form and the linear form is a unipotent group and so the corresponding action is subject to Ratner's theorems. As in the case of Dani's result \cite{Dani00} referred to above, if the  plane $\{x\in\RR^{3}\;|\; L(x)=0\}$ intersects the cone $\{x\in\RR^{3}\;|\; Q(x)=0\}$ transversally, we expect that the density will fail for a full Hausdorff dimension set of pairs.  
\item Condition $(2)$ is natural to assume for density.

\end{enumerate}

Along the way, we need several lemmata which can also be used to study the Oppenheim conjecture for a single inhomogeneous quadratic form, and so we take the opportunity to present a self contained proof of the following theorem.

\begin{theorem}\label{thm:io}
	Let $Q_{\xi}$ be an indefinite, irrational and non-degenerate quadratic form in $n$ variables, $n\geq3$. Then $Q_{\xi}(\mathbb{Z}^{n})$ is dense in $\mathbb{R}$.
\end{theorem}

As noted above, Theorem \ref{thm:io} is already implied by the work of Margulis and Mohammadi \cite{MaMo}, so we make no claims to originality as regards Theorem \ref{thm:io}. 

\section{Notation}
This paper is heavy on notation, so we are devoting this section to defining the various groups that will play a role in subsequent chapters. We have a natural action of $\SL(3,\RR)\ltimes \RR^{3}$ on $\RR^{3}$ given by $$(g,v).x=gx+v$$ where $(g,v)\in \SL(3,\RR)\ltimes \RR^{3}$ and $x\in \RR^{3}$.
\begin{definition}
	Given inhomogeneous quadratic forms $Q_{\xi}$ and $Q'_{\xi'}$ on $\RR^{3}$, say $Q_{\xi}$ is equivalent to $Q'_{\xi'}$ denoted by $Q_{\xi} \sim Q'_{\xi '}$ iff there exists $(g,v)\in \SL(3,\mathbb{R})\ltimes \mathbb{R}^{3}$ and $\lambda \in \mathbb{R}\setminus \{0\}$ such that $\lambda Q_{\xi}((g,v).x)=Q'_{\xi '}(x) \; \forall \, x\in \mathbb{R}^{3}$.
\end{definition}

Given an inhomogeneous, indefinite and nondegenerate quadratic form $Q_{\xi}$, it is easy to see that $Q_{\xi} \sim Q_{0}$, where $Q_{0}(x)=x_{1}^{2}+x_{2}^{2}-x_{3}^{2}$. Indeed, since $Q$ is an indefinite, nondegenerate and homogeneous quadratic form in 3 variables, its signature is either $(2,1)$ or $(1,2)$ and hence there exists $\lambda \in \mathbb{R}\setminus \{0\}$ and $g\in \SL(3,\mathbb{R})$ such that $\lambda Q(gx)=Q_{0}(x)$. Let $v=-\xi$. Then, $\lambda Q_{\xi}((g,v).x)=\lambda Q(gx)=Q_{0}(x)$ which gives $Q_{\xi} \sim Q_{0}$.

\begin{definition}
	Let $Q_{\xi}, Q'_{\xi'}$ be inhomogeneous quadratic forms and $L,L'$ be linear forms on $\RR^{3}$. We say that the pairs $(Q_{\xi},L)$ and $(Q'_{\xi'},L')$ are equivalent iff there exists $\lambda , \mu \in \RR\ \setminus \{0\}$ and $(g,v)\in \SL(3,\RR)\ltimes \RR^{3}$ such that $\lambda Q_{\xi}((g,v).x)=Q'_{\xi'}(x)$ and $\mu L((g,v).x)=L'(x)$.
\end{definition}

\begin{definition}\label{def:sta}
	For an inhomogeneous quadratic form $Q_{\xi}$ and a linear form $L$ on $\RR^{3}$, define 
	\begin{equation*}
		\SO(Q_{\xi}):=\{(g,v)\in \SL(3,\mathbb{R})\ltimes \mathbb{R}^{3}\; | \;Q_{\xi}((g,v).x)=Q_{\xi}(x)\; \forall \; x\in \mathbb{R}^{3} \},
	\end{equation*}
	\begin{equation*}
		\SO(L):=\{(g,v)\in \SL(3,\mathbb{R})\ltimes \mathbb{R}^{3}\; | \;L((g,v).x)=L(x)\; \forall \; x\in \mathbb{R}^{3} \},
	\end{equation*}
	\noindent and
	\begin{equation*}
		\SO(Q_{\xi},L)= \SO(Q_{\xi})\cap \SO(L).
	\end{equation*}
\end{definition}

For a subgroup $H$ of $G$, $H^{\circ}$ denotes the identity component of $H$ and $N(H)$ denotes the normalizer of $H$ in $G$.  We set $G=\SL(3,\mathbb{R})\ltimes \mathbb{R}^{3}, \Gamma=\SL(3,\mathbb{Z})\ltimes \mathbb{Z}^{3}$ and $H=\SO(2,1)^{\circ}\ltimes \{0\}.$ Note that $\Gamma$ is a nonuniform lattice in $G$.\\
	 
	 Let $$V_{1}=\left\{\begin{pmatrix}
	 	1 & t &  \frac{t^{2}}{2}  \\
	 	0 & 1 &  t \\
	 	0 & 0 & 1
	 	\end{pmatrix} : t\in \RR\right\},
	 \quad 
	 V=\left\{\begin{pmatrix}
	 1 & a &  b  \\
	 0 & 1 &  a \\
	 0 & 0 & 1
	 \end{pmatrix} : a,b\in \RR\right\},$$\\
	 $$W=\left\{\begin{pmatrix}
	 1 & a &  b \\
	 0 & 1 &  c \\
	 0 & 0 & 1
	 \end{pmatrix} : a,b,c\in \RR\right\},
	 \quad \text{ and }
	 D=\left\{\begin{pmatrix}
	 t & 0 &  0 \\
	 0 & 1 &  0\\
	 0 & 0 & t^{-1}
	 \end{pmatrix} : t\in \RR\setminus \{0\}\right\}.$$\\
	 For $t\in \RR$, let 
	 $$v(t)=\begin{pmatrix}
	 1 & 0 &  t \\
	 0 & 1 &  0 \\
	 0 & 0 & 1
	 \end{pmatrix},\quad
	 N=\left\{\begin{pmatrix}
	 	a & b &  c \\
	 	0 & a^{-2} &  d\\
	 	0 & 0 & a
	 \end{pmatrix} : a\in \RR\setminus \{0\}, b,c,d \in \RR\right\},$$\\
	 $$N_{1}=\left\{\begin{pmatrix}
	 a^{2} & b &  c \\
	 0 & a^{-5} &  d\\
	 0 & 0 & a^{3}
	 \end{pmatrix} : a\in \RR\setminus \{0\}, b,c,d \in \RR\right\},$$
	 $$N_{2}=\left\{\begin{pmatrix}
	 a^{3} & b &  c \\
	 0 & a^{-5} &  d\\
	 0 & 0 & a^{2}
	 \end{pmatrix} : a\in \RR\setminus \{0\}, b,c,d \in \RR\right\}.$$\\
	 Let 
	 $$Q_{1}=\left\{g=\begin{pmatrix}
	 1 &     \\
	 0 &  A  \\
	 0 &  
	 \end{pmatrix} : A \text{ is a } 3\times 2 \text{ matrix such that } g\in \SL(3,\RR)\right\}$$\\
	 and
	 $$Q_{2}=\left\{g=\begin{pmatrix}
	  & B &   \\
	 0 & 0 &  1
	 	 \end{pmatrix} : B \text{ is a } 2\times 3 \text{ matrix such that } g\in \SL(3,\RR)\right\}.$$\\
For $\beta \in\RR\setminus\{0\}$,  we set $$P_{\beta}=\left\{\left(\begin{pmatrix}
	 	 1 & t &  c \\
	 	 0 & 1 &  t\\
	 	 0 & 0 & 1
	 	 \end{pmatrix},\begin{pmatrix}
	 	 a \\
	 	 b \\
	 	 (c-\frac{t^{2}}{2})a\beta 
	 	 \end{pmatrix} \right): a,b,c,t \in\RR \right\},$$\\
and for $\alpha\in \RR$, set $$H_{\alpha}=\left\{\left(\begin{pmatrix}
	 	 1 & t &  \frac{t^{2}}{2} \\
	 	 0 & 1 &  t\\
	 	 0 & 0 & 1
	 	 \end{pmatrix},\begin{pmatrix}
	 	 \frac{\alpha t^{2}}{2} \\
	 	 \alpha t\\
	 	 0 
	 	 \end{pmatrix} \right): t \in\RR \right\},$$ \\
$$A_{\alpha}=\left\{\left(\begin{pmatrix}
	 	 1 & a &  b \\
	 	 0 & 1 &  c\\
	 	 0 & 0 & 1
	 	 \end{pmatrix},\begin{pmatrix}
	 	 \frac{a^{2}\alpha}{2} \\
	 	 a\alpha \\
	 	 0 
	 	 \end{pmatrix} \right): a,b,c, \in\RR \right\},$$
and	 	 
	 	$$B_{\alpha}=\left\{\left(\begin{pmatrix}
	 	 1 & a &  b \\
	 	 0 & 1 &  c\\
	 	 0 & 0 & 1
	 	 \end{pmatrix},\begin{pmatrix}
	 	 d \\
	 	 a\alpha \\
	 	 0 
	 	 \end{pmatrix} \right): a,b,c,d \in\RR \right\}.$$\\
Note that $N(V_{1})=DV$ and $N(V)=DW$ and\\
		 $$N(W)=\left\{g=\begin{pmatrix}
		 d & a &  b  \\
		 0 & e &  c\\
		 0 & 0 & f
		 \end{pmatrix}:g \in \SL(3,\RR)\right\},
		  N(Q_{2})=\left\{g=\begin{pmatrix}
		 d & a &  b  \\
		 g & e &  c\\
		 0 & 0 & f
		 \end{pmatrix}:g \in \SL(3,\RR)\right\}.$$\\
We now move to the Lie algebras of these subgroups.		 
Let $$\mathfrak{V_{1}}=\left\{\begin{pmatrix}
		 0 & a &  0  \\
		 0 & 0 &  a \\
		 0 & 0 & 0
		 \end{pmatrix} : a,\in \RR\right\},\quad
		 \mathfrak{V}=\left\{\begin{pmatrix}
		 0 & a &  b  \\
		 0 & 0 &  a \\
		 0 & 0 & 0
		 \end{pmatrix} : a,b\in \RR\right\},$$\\
		 $$\mathfrak{W}=\left\{\begin{pmatrix}
		 0 & a &  b  \\
		 0 & 0 &  c \\
		 0 & 0 & 0
		 \end{pmatrix} : a,b,c\in \RR\right\},\quad
		 \mathfrak{N}=\left\{\begin{pmatrix}
		 a & b &  c  \\
		 0 & -2a &  d \\
		 0 & 0 & a
		 \end{pmatrix} : a,b,c,d\in \RR\right\},$$\\
		 $$\mathfrak{N_{1}}=\left\{\begin{pmatrix}
		 2a & b &  c  \\
		 0 & -5a &  d\\
		 0 & 0 & 3a
		 \end{pmatrix} : a,b,c,d\in \RR\right\},\quad
		 \mathfrak{N_{2}}=\left\{\begin{pmatrix}
		 3a & b &  c  \\
		 0 & -5a &  d\\
		 0 & 0 & 2a
		 \end{pmatrix} : a,b,c,d\in \RR\right\},$$\\
		 $$\mathfrak{Q_{1}}=\left\{\begin{pmatrix}
		 0 & b &  c  \\
		 0 & a &  d\\
		 0 & e &-a
		 \end{pmatrix} : a,b,c,d,e\in \RR\right\},\quad
		 \mathfrak{Q_{2}}=\left\{\begin{pmatrix}
		 a & b &  c  \\
		 e & -a &  d\\
		 0 & 0 & 0
		 \end{pmatrix} : a,b,c,d,e\in \RR\right\}.$$\\
		 For $t\in \RR$, let $$\mathfrak{R_{t}}=\left\{\begin{pmatrix}
		 a & b &  0  \\
		 c & 0 &  b-2tc\\
		 0 & c & -a
		 \end{pmatrix} : a,b,c,\in \RR\right\},$$\\
		 and for $\beta \in\RR\setminus\{0\}$, set 
		 $$\mathfrak{P_{\beta}}=\left\{\left(\begin{pmatrix}
		 0 & a &  b \\
		 0 & 0 &  a\\
		 0 & 0 & 0
		 \end{pmatrix},\begin{pmatrix}
		 c \\
		 d \\
		 b\beta
		 \end{pmatrix} \right): a,b,c,t \in\RR \right\}.$$\\
		 
		 For $\alpha\in \RR$, set $$\mathfrak{H_{\alpha}}=\left\{\left(\begin{pmatrix}
		 0 & a &  0 \\
		 0 & 0 & a\\
		 0 & 0 & 0
		 \end{pmatrix},\begin{pmatrix}
		 0 \\
		 a\alpha\\
		 0 
		 \end{pmatrix} \right): t \in\RR \right\}.$$ \\
		 
		 Finally, we set $$\mathfrak{A}_{\alpha}=\left\{\left(\begin{pmatrix}
		 0 & a &  b \\
		 0 & 0 &  c\\
		 0 & 0 & 0
		 \end{pmatrix},\begin{pmatrix}
		 0 \\
		 a\alpha \\
		 0 
		 \end{pmatrix} \right): a,b,c, \in\RR \right\},$$ \\
		 and
		 $$\mathfrak{B}_{\alpha}=\left\{\left(\begin{pmatrix}
		 0 & a &  b \\
		 0 & 0 &  c\\
		 0 & 0 & 0
		 \end{pmatrix},\begin{pmatrix}
		 d \\
		 a\alpha \\
		 0 
		 \end{pmatrix} \right): a,b,c,d \in\RR \right\}.$$\\

	 For a subgroup $C$ of $\SL(3,\RR)$, denote by $C \ltimes \RR$ and $C \ltimes \RR^{2}$ the subgroups of $G$ consisting of elements 
	 $$\left\{\left(h,\begin{pmatrix}
	 a \\
	 0 \\
	 0 
	 \end{pmatrix}\right): a\in \RR, h\in C\right\}, \text{ and } \left\{\left(h,\begin{pmatrix}
	 a \\
	 b \\
	 0 
	 \end{pmatrix}\right): a,b\in \RR, h\in C\right\}$$ respectively.

\section{Preparatory Lemmata}

In this section, we prove some lemmata required for proving the theorems.
\begin{lemma}\label{l}
With notation as in (\ref{def:sta}), we have that
	$$\SO(Q_{\xi},L)=\{(g,g\xi -\xi)\; | \; g\in  \SO(Q,L) \} .$$ 
\end{lemma}
\begin{proof}
   It is easy to see that for $g\in \SO(Q,L)$, $(g,g\xi-\xi)\in \SO(Q_{\xi},L)$. Conversely, suppose $(g,v)\in \SO(Q_{\xi},L)$. Then for $x\in \RR^{3}$, $Q_{\xi}\left((g,v).x\right)=Q_{\xi}(x)$ and $L((g,v).x)=L(x)$ which implies that $Q(gx+v+\xi)=Q(x+\xi)$ and $L(gx+v)=L(x)$. This gives that 
   $$Q(gy+v+\xi-g\xi)=Q(y)\; \forall \; y\in \RR^{3}.$$ 
   Let $\xi'=v+\xi-g\xi$. Then for every $y\in \RR^{3}$, $Q(gy+\xi')=Q(y)$ which implies that $$Q(gy)+Q(\xi')+2(gy)^{t}A\xi'=Q(y),$$ where $A$ denotes the symmetric matrix corresponding to the quadratic form $Q$. This gives that for every $ y\in \RR^{3}$, $Q(gy)=Q(y)$ and $(gy)^{t}A\xi'=0$ which further shows that $g\in \SO(Q)$ and $\xi'=v+\xi-g\xi=0$. Therefore $v=g\xi-\xi$. Substituting for $v$ in $L(gx+v)=L(x)$ gives $L(gy)=L(y)\;\forall \; y\in \RR^{3}$. Hence $g\in \SO(Q,L)$ and $v=g\xi-\xi$ thus proving the lemma.
   \end{proof}
   \begin{remark}\label{ll}
    Taking $L=0$ in the above lemma gives $\SO(Q_{\xi})=\{(g,g\xi -\xi)\; | \; g\in  \SO(Q) \}$. 
   \end{remark}

\begin{lemma}\label{l1}
	With notation as above, 
	$$\SO(Q_{\xi})^{\circ}= (g,-\xi)H(g,-\xi)^{-1}$$ 
	where $g \in \SL(3, \RR)$ is such that $\lambda Q(gx)=Q_{0}(x)$ for some $\lambda \in \RR\setminus \{0\}$.
\end{lemma}
\begin{proof}
	Since $\lambda Q(gx)=Q_{0}(x)$, we have that $\SO(Q)=g\SO(2,1)g^{-1}$. Let $h\in \SO(2,1)$. It is straight forward to compute that $(g,-\xi)(h,0)(g,-\xi)^{-1}=(ghg^{-1},ghg^{-1}\xi-\xi)$. Then,
	\begin{align*}
		Q_{\xi}\left((ghg^{-1},ghg^{-1}\xi-\xi).x\right)&=Q\left(ghg^{-1}(x+\xi)\right)\\
		&=Q(x+\xi) \quad [\because ghg^{-1}\in \SO(Q)]\\
		&=Q_{\xi}(x).
	\end{align*}
Therefore $(g,-\xi)\SO(2,1)\ltimes \{0\}(g,-\xi)^{-1}\subseteq \SO(Q_{\xi})$.\\
Now, let $(g',v)\in \SO(Q_{\xi})$. By Remark \ref{ll}, we get that $g'\in \SO(Q)$ and $v=g'\xi-\xi$. Since $\SO(Q)=g\SO(2,1)g^{-1}$, there exists $h\in \SO(2,1)$ such that $g'=ghg^{-1}$. Therefore $$(g',v)=(ghg^{-1},ghg^{-1}\xi-\xi)=(g,-\xi)(h,0)(g,-\xi)^{-1}.$$ Hence,
$$\SO(Q_{\xi})\subseteq (g,-\xi)\SO(2,1)\ltimes \{0\}(g,-\xi)^{-1}.$$ Taking the identity components, we get that $\SO(Q_{\xi})^{\circ}=(g,-\xi)H(g,-\xi)^{-1}$.
\end{proof}
\begin{lemma}\label{l2}
$H$ is generated by unipotent elements.
\end{lemma}
\begin{proof}
  Let $\h$ be the lie algebra of $H$. Denote by 
  $$h_{(a,b,c)}=\left( \begin{pmatrix}
  0&a&b\\
  -a&0&c\\
  b&c&0
  \end{pmatrix},
  \begin{pmatrix}
  0\\
  0\\
  0
  \end{pmatrix}\right).$$\\
  Then $\h=\so \ltimes \{0\}=\{h_{(a,b,c)}\;|\; a,b,c\in \RR\}$. The elements $h_{(1,1,0)}$ and $h_{(1,0,1)}$ of $\h$ are nilpotent and their Lie bracket is $h_{(1,1,1)}$. Since $h_{(1,1,0)}, h_{(1,0,1)}$ and $h_{(1,1,1)}$ form a basis for $\h$, we get that $h_{(1,1,0)}$ and $h_{(1,0,1)}$ generate the Lie algebra $\h$. As $\h$ is generated by nilpotent elements and $H$ is connected, we get that $H$ is generated by unipotent elements.
\end{proof}
\noindent For a subset $S$ of $G$, we denote by $\overline{\overline{S}}$ its Zariski closure.
\begin{lemma}\label{l3}
	Let $S$ be a subset of $\SL(3,\ZZ)\ltimes \ZZ^{3}$. Then $\overline{\overline{S}}$ is defined over $\QQ$.
\end{lemma}
\begin{proof}
	Suppose $\overline{\overline{S}}$ is the set of zeroes of $\mathfrak{S}$ for some $\mathfrak{S}\subset \mathfrak{P}^{n}$, where $\mathfrak{P}^{n}$ denotes the set of polynomials of degree $\leq n$. Then the subspace $\{f\in \mathfrak{P}^{n}\;|\; f(S)=0\}$ is defined by linear equations with rational coefficients, since $S\subset \SL(3,\ZZ)\ltimes \ZZ^{3}$. As $\mathfrak{S}\subseteq \{f\in\mathfrak{P}^{n}\; |\; f(S)=0 \}$, we get that $\overline{\overline{S}}$ is defined over $\QQ$.
\end{proof}
\begin{lemma}\label{l4}
	Let $Q$ be an indefinite and nondegenerate quadratic form. If $\SO(Q_{\xi})^{\circ}$ is defined over $\QQ$, then $Q_{\xi}$ is not an irrational quadratic form.
\end{lemma}
\begin{proof}
	Firstly, we will show that if $\SO(Q_{\xi})^{\circ}= \SO(Q'_{\xi'})^{\circ}$, then $\xi=\xi'$ and there exists $c\in \RR$ such that $\sigma=c\sigma'$ where $\sigma$ and $\sigma'$ are the symmetric matrices corresponding to $Q$ and $Q'$ respectively. Let $(h,v)\in \SO(Q_{\xi})^{\circ}$. Then by Remark $\ref{ll}$, $h\in \SO(Q)$ and $v=h\xi-\xi$. Since $\SO(Q_{\xi})^{\circ}= \SO(Q'_{\xi'})^{\circ}$, $(h,v)$ also lies in $\SO(Q'_{\xi'})^{\circ}$ which implies that $h\in \SO(Q')$ and $v=h\xi'-\xi'$. Consider,
	\begin{align*}
(h,v)(\sigma,-\xi)(\sigma',-\xi')^{-1}(h,v)^{-1}&=(h,v)(\sigma,-\xi)(h^{t},0)(h^{t},0)^{-1}(\sigma',-\xi')^{-1}(h,v)^{-1}\\
		&=(h\sigma h^{t},v-h\xi)(h\sigma' h^{t}, v-h\xi')^{-1}\\
		&=(\sigma,\xi)(\sigma',-\xi')^{-1}.
	\end{align*}
This implies that $(\sigma,-\xi)(\sigma',-\xi')^{-1}$ lie in the centralizer of $\SO(Q_{\xi})^{\circ}$.\\
\noindent We now\\
$\mathbf{Claim}$: The centralizer of $\SO(Q_{\xi})^{\circ}$ in $\GL(3,\RR)\ltimes \RR^{3}$ is $\left\{(c\mathbf{I},(c-1)\xi)\;|\;c\in \RR \setminus \{0\}\right\}$ where $\mathbf{I}$ denotes the identity matrix.\\
\noindent Proof of the claim. Let $(A,v)\in \GL(3,\RR)\ltimes \RR^{3}$ be such that $(A,v)$ commutes with every element of $H$. Then $A$ lies in the centralizer of $\SO(2,1)^{\circ}$ and $hv = v$ for every $h\in \SO(2,1)^{\circ}$. From (Lemma 2.2 (\romannumeral 2\relax), chapter 6, \cite{BM}), it follows that $A=c\mathbf{I}$ for some $c\in \RR$ and $v=0$. Therefore, the centralizer of $H$ is $\{(c\mathbf{I},0)\;|\;c\in \RR\setminus\{0\}\}$. Since $Q$ is indefinite and nondegenerate, there exists $\lambda\in\RR\setminus\{0\}$ and $g\in \SL(3,\RR)$ such that $\lambda Q(gx) = Q_{0}(x)$. Hence by Lemma $\ref{l1}$, $\SO(Q_{\xi})^{\circ}= (g,-\xi)H(g,-\xi)^{-1}$. Therefore, the centralizer of $\SO(Q_{\xi})^{\circ}$ is $$\{(g,-\xi)(c\mathbf{I},0)(g,-\xi)^{-1}\;|\;c\in \RR\setminus\{0\}\}=\{(c\mathbf{I},(c-1)\xi)\;|\; c\in \RR\setminus\{0\}\}$$ thereby proving the claim.\\
Therefore, there exists $c\in \RR\setminus \{0\}$ such that $(\sigma,-\xi)(\sigma',-\xi')^{-1}=(c\mathbf{I},(c-1)\xi)$ which gives that $\sigma=c\sigma'$. Since $\SO(Q_{\xi})^{\circ}= \SO(Q'_{\xi'})^{\circ}$, the claim implies that $\xi=\xi'$.\\
Now, let $\phi \in \Aut(\mathbb{C}/\QQ)$. By $\phi(Q)$ we mean the quadratic form obtained by applying $\phi$ to the coefficients of $Q$ and $\phi(\xi)$ is the vector obtained by applying $\phi$ to each coordinate of $\xi$. Then $\SO\left(\phi(Q)_{\phi(\xi)}\right)^{\circ}=\phi\left(\SO(Q_{\xi})^{\circ}\right)=\SO(Q_{\xi})^{\circ}$ (Since $\SO(Q_{\xi})^{\circ}$ is defined over $\QQ$). Therefore, there exists $\alpha_{\phi}\in \RR\setminus\{0\}$ such that $\phi(\sigma)=\alpha_{\phi}\sigma$ and $\phi(\xi)=\xi$ where $\sigma$ is the matrix corresponding to the quadratic form $Q$. By taking a scalar multiple, we can assume that one of the matrix entries of $\sigma$ is rational. Then, as $\phi$ fixes that coefficient, we get that $\alpha_{\phi}=1$. Hence $\phi(\sigma)=\sigma$ and $\phi(\xi)=\xi$ for every $\phi \in \Aut(\mathbb{C}/\QQ)$. Since the fixed point set of $\Aut(\mathbb{C}/\QQ)$ is $\QQ$, we get that $Q$ is a scalar multiple of a rational form and $\xi$ is a  rational vector thus proving that $Q_{\xi}$ is not an irrational quadratic form.
\end{proof}

The following Lemma is well known, see for instance (exercise 17, 1.2, \cite{WM}). We will need it so we present a proof here for completeness.
\begin{lemma}\label{l6}
	$\so$ is a maximal Lie subalgebra of $\sL$.
\end{lemma}
\begin{proof}
 Suppose there exists a subalgebra $\h$ such that $\so\subsetneq\h\subsetneq\sL$. Consider the adjoint representation of $\sL$ restricted to $\so$. This gives a representation of $\so$ and $\h$ is an $\ad(\so)$-invariant subspace of $\sL$ since $\so \subsetneq \h$. Since $\so$ is isomorphic to $\mathfrak{sl}(2,\RR)$, by classification of representations of $\mathfrak{sl}(2,\RR)$ (Proposition 4.9.22, \cite{WM}), we get that there is a sequence $\lambda_{1},\dots,\lambda_{n}$ of natural numbers and a basis $\{w_{i,j}\;|\; 1\leq i\leq n,\; 0\leq j\leq \lambda_{i}\}$ of $\sL$ such that for all $i, j$ we have
 \begin{enumerate}
 \item $\ad(\underline{a})(w_{i,j})=(2j-\lambda_{i})w_{i,j}$\\
 \item $\ad(\underline{u})(w_{i,j})=(\lambda_{i}-j)w_{i,j+1}$\\
 \item $\ad(\underline{v})(w_{i,j})=jw_{i,j-1}$
 \end{enumerate}
 where $\underline{a}, \underline{u}$ and $\underline{v}$ form a basis of $\so$ satisfying the relations $$[\underline{u},\underline{a}]=2\underline{u}, [\underline{v},\underline{a}]=-2\underline{v} \text{ and } [\underline{v},\underline{u}]=\underline{a}.$$ By (2), it follows that $\ke$ is spanned by $\{w_{1,\lambda_{1}},...,w_{n,\lambda_{n}}\}$. So if $W$ is an invariant subspace of $\sL$ containing $\ke$, then $W=\sL$ by (3). Therefore $\h$ does not contain $\ke$. This implies that $\h\cap \ke \subsetneq\ke$.\\
 By replacing $\sL$ by $\h$, we may consider $\ad_{\h}: \so\rightarrow \End(\h)$. Now, $\ker(\ad_{\h}(\underline{u}))=\h \cap \ke$. As $\so$ is a proper $\ad_{\h}(\so)$ invariant subspace of $\h$, by the same argument as above we get that $\h\cap \ke$ is not contained in $\so$. This implies that 
 \begin{equation}
 	 \so \cap \ke \subsetneq \h \cap \ke \subsetneq \ke.
 	 \end{equation}
Let  $$\underline{a}=
\begin{pmatrix}
0&2&2\\
-2&0&2\\
2&2&0
\end{pmatrix},
\underline{u}=
\begin{pmatrix}
0&\sqrt{2}&\sqrt{2}\\
-\sqrt{2}&0&0\\
\sqrt{2}&0&0
\end{pmatrix} \text{ and }
\underline{v}=
\begin{pmatrix}
0&-\sqrt{2}&0\\
\sqrt{2}&0&-\sqrt{2}\\
0&-\sqrt{2}&0
\end{pmatrix}.$$
Then $\{ \underline{a},\underline{u},\underline{v}\} $ forms a basis of $\so$ satisfying $[\underline{u},\underline{a}]=2\underline{u}$, $[\underline{v},\underline{a}]=-2\underline{v}$ and $[\underline{v},\underline{u}]=\underline{a}$. Now, ker(ad(\underline{u}))=$\left\{\begin{pmatrix}
	0&a&a\\
	-a&b&b\\
	a&-b&-b
\end{pmatrix}: a,b \in \RR \right\}$ which is a subspace of dimension 2. From (3.1) it follows that $\so \cap \ke =\{0\}$ which is a contradiction since $$\underline{u}\in \so\cap \ke.$$ 
\end{proof}
\begin{lemma}\label{l7}
The only closed connected subgroups of $G$ containing $H$ are $H$, $\SO(2,1)^{\circ}\ltimes \mathbb{R}^{3}$, $\SL(3,\mathbb{R})\ltimes \{0\}$ and $G$.
\end{lemma}
\begin{proof}
Denote by $\mathfrak{g}$ and $\mathfrak{h},$ the Lie algebras of $G$ and $H$ respectively. We will show that the only Lie subalgebras of $\g$ containing $\h$ are $\h$, $\sor$, $\sL \ltimes \{0\}$ and $\g$. The Lemma will follow from the correspondence between Lie groups and Lie algebras. Let $\mathfrak{f}$ be a Lie subalgebra of $\mathfrak{g}$ such that $\mathfrak{h}\subsetneq \mathfrak{f} \subsetneq \mathfrak{g}$. Let $P$ be the projection map from $\mathfrak{g}$ to $\mathfrak{sl}(3,\mathbb{R})$. Then, $P(\mathfrak{f})$ is a lie subalgebra of $\mathfrak{sl}(3,\mathbb{R})$ containing $\mathfrak{so}(2,1)$. Since $\mathfrak{so}(2,1)$ is a maximal Lie subalgebra of $\mathfrak{sl}(3,\mathbb{R})$ (by Lemma $\ref{l6}$), $P(\mathfrak{f})$ is either equal to $\mathfrak{sl}(3,\mathbb{R})$ or $\mathfrak{so}(2,1)$. We examine these cases separately.\\

\noindent Case 1: $P(\mathfrak{f})=\mathfrak{so}(2,1) $.\\
Since $\mathfrak{so}(2,1) \ltimes \{0\}\subsetneq \mathfrak{f}$, there exists an element $(g,v)\in \mathfrak{f}$ such that $(g,v)\notin \mathfrak{h}$.  The assumption $P(\mathfrak{f})=\mathfrak{so}(2,1)$ implies that $g \in \mathfrak{so}(2,1)$. Since $(g,v)\notin \mathfrak{h}$, we have that $v\neq 0$. As $(g,0)\in \f$, we get $(g,v)-(g,0)=(0,v)\in \f$. Therefore, for all $g\in \so$, $$[(g,0),(0,v)]=(0,gv)\in \f.$$ Since $\so$ acts irreducibly on $\RR^{3}$, we get that $(0,w)\in \f, \;\forall \;w\in \RR^{3}$. Hence, $\forall \; g\in \so$ and $\forall \; w\in \RR^{3}$, we have that $$(g,0)+(0,w)=(g,w)\in \f.$$ Therefore, $\sor \subset \f$ and since $P(\f)=\so$, we get that $\f=\sor$. \\

\noindent Case 2: $P(\f)=\sL$.\\
Assume that for some $g\in \sL \setminus \so$, we have that $(g,0)\in \f$. Since the Lie subalgebra generated by $\so$ and $g$ is $\sL$ (as $\so$ is a maximal subalgebra of $\sL$), we get $(g,0)\in \f \;\forall\; g\in \sL$. Therefore, $\sL \ltimes \{0\}\subset \f$. If $\sL \ltimes \{0\}\subsetneq \f$ then $(0,v)\in \f$ for some non-zero $v$ which implies $(0,w)\in \f \;\forall\; w\in \RR^{3}$ as in Case 1 and hence $\f=\slr$ which is a contradiction. Therefore, $\f=\sL\ltimes \{0\}$.\\

Now, suppose for every $g\in \sL \setminus \so$, we have that $(g,0)\notin \f$. Then for every $g\in \sL \setminus \so$, there exists a non-zero element $v_{g}\in \RR^{3}$, such that $(g,v_{g})\in \f$ since $P(\f)=\sL$. Let
\begin{center}
$g_{1}=
\begin{pmatrix}
	1 & 0 &  0  \\
	0 & 1 &  0 \\
	0 & 0 & -2
\end{pmatrix},\quad
g_{2}=
\begin{pmatrix}
1 & 0 &  1  \\
0 & 1 &  0 \\
1 & 0 & -2
\end{pmatrix},\quad
g=
\begin{pmatrix}
	1 & 0 &  1  \\
	6 & 1 &  0 \\
	1 & 0 & -2
\end{pmatrix}$
\end{center}
\begin{center}
$h=
\begin{pmatrix}
0 & 1 &  0  \\
-1 & 0 &  0 \\
0 & 0 & 0
\end{pmatrix},\quad
k=
\begin{pmatrix}
0 & 0 &  1  \\
0 & 0 &  -1 \\
1 & -1 & 0
\end{pmatrix}.
$
\end{center}
Since $g_{1},\,g_{2},\, g \in \sL\setminus\so$, there exist non zero elements $v_{g_{1}},\, v_{g_{2}},\,v_{g}$ in $\RR^{3}$ such that $$(g_{1},v_{g_{1}}),\, (g_{2},v_{g_{2}}),\, (g,v_{g})\in \f.$$ Since $h,k \in \so$ we have that $(h,0),\, (k,0)\in \f$. Therefore,\\
$$\begin{bmatrix}
(h,0),(g_{1},v_{g_{1}})
\end{bmatrix},
\begin{bmatrix}
(h,0),(g_{2},v_{g_{2}})
\end{bmatrix},
\begin{bmatrix}
(k,0),(g,v_{g})
\end{bmatrix} \in \f$$ which implies that 
$$
\begin{pmatrix}
[h,g_{1}], hv_{g_{1}}
\end{pmatrix},\\
\begin{pmatrix}
[h,g_{2}], hv_{g_{2}}
\end{pmatrix},
\begin{pmatrix}
[k,g], kv_{g}
\end{pmatrix}\in \f
.$$ 
It is straight forward to check that $[h,g_{1}],\, [h,g_{2}], \, [k,g]\in \so$. Therefore $$(0,hv_{g_{1}}), (0,hv_{g_{2}}), (0,kv_{g})\in \f.$$ Now let
$$
v_{g_{1}}=
\begin{pmatrix}
a_{1}\\
b_{1}\\
c_{1}
\end{pmatrix},
v_{g_{2}}=
\begin{pmatrix}
a_{2}\\
b_{2}\\
c_{2}
\end{pmatrix},
v_{g}=
\begin{pmatrix}
a_{3}\\
b_{3}\\
c_{3}
\end{pmatrix},
$$ 
then
$$
hv_{g_{1}}=
\begin{pmatrix}
b_{1}\\
-a_{1}\\
0
\end{pmatrix},
hv_{g_{2}}=
\begin{pmatrix}
b_{2}\\
-a_{2}\\
0
\end{pmatrix}, 
kv_{g}=
\begin{pmatrix}
c_{3}\\
-c_{3}\\
a_{3}-b_{3}
\end{pmatrix}
.$$
If one among $a_{1},b_{1},a_{2},b_{2}$ is non-zero, then either $hv_{g_{1}}$ or $hv_{g_{2}}$ is non-zero and hence $(0,v)\in \f$ for some non zero element $v$. As in case 1, this implies that $\sor \subset \f$. Similarly, if either $c_{3} $ is non-zero or $a_{3}\neq b_{3}$, then $kv_{g}$ is non-zero and hence $(0,v) \in \f$ for some non-zero $v$ which again implies $\sor \subset \f$. Since $P(\f)=\sL$, we get $\f=\slr$ which is a contradiction.
Now, suppose $a_{1}=b_{1}=a_{2}=b_{2}=c_{3}=0$ and $a_{3}=b_{3}$ then 
$
v_{g_{1}}=
\begin{pmatrix}
0\\
0\\
c_{1}
\end{pmatrix},
v_{g_{2}}=
\begin{pmatrix}
0\\
0\\
c_{2}
\end{pmatrix},$ and
$v_{g}=
\begin{pmatrix}
a_{3}\\
a_{3}\\
0
\end{pmatrix}
$.
It is easy to compute that 
$$\begin{bmatrix}
(g,v_{g}),(g_{1},v_{g_{1}})
\end{bmatrix}=
\left(
\begin{pmatrix}
0 & 0 & -3\\
0 & 0 & 0\\
3 & 0 &0
\end{pmatrix},
\begin{pmatrix}
c_{1}-a_{3}\\
-a_{3}\\
-2c_{1}
\end{pmatrix}
\right) \in \f,
$$ and\\
$$\begin{bmatrix}
(g_{2},v_{g_{2}}),(g_{1},v_{g_{1}})
\end{bmatrix}=
\left(
\begin{pmatrix}
0 & 0 & -3\\
0 & 0 & 0\\
3 & 0 &0
\end{pmatrix},
\begin{pmatrix}
c_{1}\\
0\\
2(c_{2}-c_{1})
\end{pmatrix}
\right) \in \f
.$$ 
Hence their difference, which is $(0,v)$ for some non zero $v$, lies in $\f$. This implies that $\sor \subset \f$ which again gives $\f=\slr$ since $P(\f)=\sL$, a contradiction.
 \end{proof}
\section{Proof of Theorem \ref{thm:io}}
In this section, we prove Theorem \ref{thm:io}, i.e. the Oppenheim conjecture for inhomogeneous forms. 

\subsection*{Reduction to the case of $n=3$} \

Using induction on $n$, it follows from the following Lemma that it is enough to prove the theorem for the case of $n=3$.
\begin{lemma}
	Let $Q_{\xi}$ be an indefinite, irrational and nondegenerate quadratic form in $n$ variables, $n\geq3$. Then there exists a rational hyperplane $L$ such that the restriction of $Q_{\xi}$ to $L$ is indefinite, irrational and nondegenerate.
\end{lemma}
\begin{proof}
	Since $Q_{\xi}$ is irrational, either $Q$ is an irrational quadratic form or $\xi$ is an irrational vector. Firstly, assume that $Q$ is irrational. Then from (Lemma 2.1, chapter 6, \cite{BM}), it follows that there exists a rational hyperplane $L$ such that restriction of $Q$ to $L$ is indefinite, irrational and nondegenerate. This implies that restriction of $Q_{\xi}$ to $L$ is indefinite, irrational and nondegenerate.\\
	Now, assume that $Q$ is not irrational. Then $\xi$ has to be an irrational vector. Since $Q$ is indefinite and nondegenerate, we can find a rational hyperplane $L$ such that restriction of $Q$ to $L$ is indefinite and nondegenerate (This is a part of the proof of (Lemma 2.1, chapter 6, \cite{BM})). Then the restriction of $Q_{\xi}$ to $L$ is irrational (Since $\xi $ is irrational), indefinite and nondegenerate. 
\end{proof}

Using the above stated lemmas, we now prove Theorem \ref{thm:io} when $n=3$.
\begin{proof}
Let $g\in \SL(3,\RR)$ and $\lambda\in \RR \setminus \{0\}$ be such that $\lambda Q(gx)=Q_{0}(x)$. By Lemma \ref{l2}, we have that $H = \SO(2,1)^{\circ}\ltimes \{0\}$ is generated by unipotent elements and since $\Gamma$ is a lattice in $G$, we may apply Ratner's orbit closure theorem \cite{Ratner} which tells us that there is a closed connected subgroup $F$ of $G$ such that 
\begin{enumerate}
\item $H\subset F,$
\item the image $[F.(g,-\xi)^{-1}]$ of $F.(g,-\xi)^{-1}$ in $G/\Gamma$ is closed and has finite $F$- invariant measure,
\item the closure of $[H.(g,-\xi)^{-1}]$ is equal to $[F.(g,-\xi)^{-1}]$.
	\end{enumerate}
By Lemma \ref{l7}, $F$ is either $H$, $\SO(2,1)^{\circ}\ltimes \RR^{3}$, $\SL(3,\RR)\ltimes \{0\}$ or $G$.\\
Case 1: Suppose $F$ is either $\SO(2,1)^{\circ}\ltimes \RR^{3}$ or $\SL(3,\RR)\ltimes \{0\}$ or $G$. Then observe that $F(g,-\xi)^{-1}\ZZ^{3}=\RR^{3}$. Hence,
\begin{align*}
	\overline{\lambda Q_{\xi}(\ZZ^{3})}\; &=\; \overline{Q_{0}\left((g,-\xi)^{-1}\ZZ^{3}\right)} \qquad[\because \lambda Q_{\xi}\left((g,-\xi)x\right)=Q_{0}(x)]\\
	&=\; \overline{Q_{0}\left((g,-\xi)^{-1}\Gamma \ZZ^{3}\right)} \qquad[\because \big(\SL(3,\ZZ)\ltimes \ZZ^{3}\big).\ZZ^{3}=\ZZ^{3}]\\
	&=\; \overline{Q_{0}\left(H(g,-\xi)^{-1}\Gamma \ZZ^{3}\right)} \qquad[\because Q_{0}(hx)=Q_{0}(x) \; \forall \; h\in H]\\
	& \supseteq \;Q_{0}\left(\overline{H(g,-\xi)^{-1}\Gamma \ZZ^{3}}\right) \qquad [\because Q_{0}\; \text{is continuous}]\\
	&=\; Q_{0}\left(F(g,-\xi)^{-1}\Gamma \ZZ^{3}\right)\qquad [\because \text{ by } (3), \overline{H(g,-\xi)^{-1}\Gamma}=F(g,-\xi)^{-1}\Gamma]\\
	&=\; Q_{0}\left(F(g,-\xi)^{-1}\ZZ^{3}\right)\\
	&=\; Q_{0}\left(\RR^{3}\right)\;=\RR.
	\end{align*}
Therefore, $Q_{\xi}(\ZZ^{3})$ is dense in $\RR$.\\
Case 2: Suppose $F=H$.\\
We will show that $Q_{\xi}$ cannot be an irrational quadratic form. By (2), $\left[H(g,-\xi)^{-1}\right]$ is closed in $G/\Gamma$ and has finite $H$-invariant measure. This implies that $(g,-\xi)H(g,-\xi)^{-1}\cap \Gamma$ is a lattice in $(g,-\xi)H(g,-\xi)^{-1}= \SO(Q_{\xi})^{\circ}$ (By Lemma $\ref{l1}$). Denote by $\Gamma_{(g,\xi)}=(g,-\xi)H(g,-\xi)^{-1}\cap \Gamma$. By the Borel density theorem, all unipotent elements of $\SO(Q_{\xi})^{\circ}$ lie in the Zariski closure of $\Gamma_{(g,\xi)}$. Since $\SO(Q_{\xi})^{\circ}$ is generated by its unipotent elements (Since it is a conjugate of $H$ and $H$ is generated by unipotent elements ), we get that $\SO(Q_{\xi})^{\circ}= \overline{\overline{\Gamma_{(g,\xi)}}}$, where $\overline{\overline{\Gamma_{(g,\xi)}}}$ denotes the Zariski closure of $\Gamma_{(g,\xi)}$. Since $\Gamma_{(g,\xi)}\subseteq \Gamma$, $\overline{\overline{\Gamma_{(g,\xi)}}}$ is defined over $\QQ$ (By Lemma $\ref{l3}$) and hence $\SO(Q_{\xi})^{\circ}$ is defined over $\QQ$. This implies that $Q_{\xi}$ is not an irrational quadratic form (By Lemma $\ref{l4}$).
\end{proof}

\section{Proof of Theorem \ref{thm:main}}
In this section, we denote by $Q_{0}$ the quadratic form defined by $Q_{0}(x)=2x_{1}x_{3}-x_{2}^{2}$.
\begin{lemma}\label{l5}
	Let $Q_{\xi}$ be an inhomogeneous, non-degenerate and indefinite quadratic form and $L$ be a linear form on $\RR^{3}$. Suppose that the plane $\{x\in\RR^{3}\;|\; L(x)=0\}$ is tangential to the cone $\{x\in\RR^{3}\;|\; Q(x)=0\}$. Then there exists $\alpha\in \RR$ such that $(Q_{\xi},L) \sim ((Q_{0})_{(0, 0,\alpha)},L_{0})$ where $Q_{0}(x)= 2x_{1}x_{3}-x_{2}^{2},$ $L_{0}(x)=x_{3}$ and $\overline{\{(Q_{\xi}(x),L(x))\;|\:x\in \RR^{3}\}}=\RR^{2}$.
	\end{lemma}
	\begin{proof}
		Since the plane $\{x\in\RR^{3}\;|\; L(x)=0\}$ is tangential to the cone $\{x\in\RR^{3}\;|\; Q(x)=0\}$, there exists $\lambda ,\mu \in \RR \setminus \{0\}$ and $g\in \SL(3,\RR)$ such that $\forall\;x\in \RR^{3}$, $\lambda Q(gx)=Q_{0}(x)$ and $\mu L(gx)=L_{0}(x)$ where $Q_{0}(x)= 2x_{1}x_{3}-x_{2}^{2}$ and $L_{0}(x)=x_{3}$. Let $\alpha= \mu L(\xi)$ and $v=g
		\begin{pmatrix}
		0\\
		0\\
		\alpha
		\end{pmatrix}-\xi$.
		Then it can be easily seen that $$\lambda Q_{\xi}((g,v).x)=(Q_{0})_{(0, 0,\alpha)}(x)$$ and $$\mu L((g,v).x)=L_{0}(x)$$ and hence $(Q_{\xi},L)$ is equivalent to $((Q_{0})_{(0, 0,\alpha)},L_{0})$. Therefore,
		$$\overline{\{(Q_{\xi}(x),L(x))\;|\:x\in \RR^{3}\}}=\overline{\{(Q_{0})_{(0,0,\alpha)}(x),L(x))\;|\:x\in \RR^{3}\}}=\RR^{2}.$$
	\end{proof} 
	
	\begin{lemma}\label{l8} With notation as in section 2,
  \begin{enumerate}
  \item the only closed connected unimodular subgroups of $G$ containing $H_{0}$ are:
\begin{center}
\begin{tabular}{ |c|c|c|c| } 
 \hline
 $V_{1}\ltimes\{0\}$ & $V_{1}\ltimes \RR$ & $V_{1}\ltimes \RR^{2}$ & $V_{1}\ltimes \RR^{3}$ \\ 
 \hline
 $V\ltimes\{0\}$ & $V\ltimes \RR$ & $V\ltimes \RR^{2}$ & $V\ltimes \RR^{3}$ \\ 
 \hline
  $W\ltimes\{0\}$ & $W\ltimes \RR$ &  $W\ltimes \RR^{2}$ &  $W\ltimes \RR^{3}$ \\
 \hline
 $v(t)\SO(Q_{0})^{\circ}v(t)^{-1}\ltimes \{0\}$ & $v(t)\SO(Q_{0})^{\circ}v(t)^{-1}\ltimes \RR^{3}$ & $Q_{1}\ltimes\{0\}$ & $Q_{1}\ltimes \RR$ \\
 \hline
 $Q_{1}\ltimes\RR^{3}$ & $Q_{2}\ltimes\{0\}$ & $Q_{2}\ltimes \RR^{2}$ & $Q_{2}\ltimes\RR^{3}$ \\ 
 \hline
 $N^{\circ}\ltimes \{0\}$ & $N^{\circ}\ltimes \RR^{3}$   & $N_{1}^{\circ}\ltimes \RR$ & $N_{2}^{\circ}\ltimes \RR^{2}$ \\ 
 \hline
 $\SL(3,\RR)\ltimes\{0\}$ &  $\SL(3,\RR)\ltimes\RR^{3}$ & $P_{\beta}$ for $\beta \in \RR\setminus\{0\}$ &  \\
 \hline 
\end{tabular}
\end{center}

  \item for $\alpha \in \RR\setminus\{0\}$, the only closed connected subgroups of $G$ containing $H_{\alpha}$ are:
  \begin{center}
\begin{tabular}{ |c|c|c|c| } 
 \hline
 $V_{1}\ltimes \RR^{2}$ & $V_{1}\ltimes \RR^{3}$ & $V\ltimes \RR^{2}$ & $V\ltimes \RR^{3}$ \\ 
 \hline
 $W\ltimes \RR^{2}$ & $W\ltimes \RR^{3}$ & $v(t)\SO(Q_{0})^{\circ}v(t)^{-1}\ltimes \RR^{3}$ & $Q_{1}\ltimes\RR^{3}$\\ 
 \hline
  $Q_{2}\ltimes \RR^{2}$ & $Q_{2}\ltimes\RR^{3}$ & $N^{\circ}\ltimes \RR^{3}$ & $N_{2}^{\circ}\ltimes \RR^{2}$\\
 \hline
 $\SL(3,\RR)\ltimes\RR^{3}$ &$A_{\alpha}$ & $B_{\alpha}$ & $P_{\beta}$ for $\beta \in \RR\setminus\{0\}$\\
 \hline
\end{tabular}
\end{center}
  
\end{enumerate}
\end{lemma}
\begin{proof}
	The lemma follows from the classification of Lie subalgebras of $\sL\ltimes\RR^{3}$. There has of course been extensive work on this subject, we use the paper \cite{pw} of Winternitz which is well suited for our purpose. Namely, by using the subalgebra classification algorithm (2.4, \cite{pw}) and from Table 1 of \cite{pw}, one can compute that the only unimodular subalgebras of $\sL\ltimes\RR^{3}$ containing $\mathfrak{H_{0}}$ are:	
	$\mathfrak{V_{1}}\ltimes\{0\},\;\mathfrak{V_{1}}\ltimes\RR,\;\mathfrak{V_{1}}\ltimes\RR^{2},\;\mathfrak{V_{1}}\ltimes\RR^{3},\;\mathfrak{V}\ltimes\{0\},\;\mathfrak{V}\ltimes\RR,\;\mathfrak{V}\ltimes\RR^{2},\;\mathfrak{V}\ltimes\RR^{3},\;\mathfrak{W}\ltimes\{0\},\;\mathfrak{W}\ltimes\RR,\;\mathfrak{W}\ltimes\RR^{2},\;\mathfrak{W}\ltimes\RR^{3},\;\mathfrak{K_{t}}\ltimes\{0\},\;\mathfrak{K_{t}}\ltimes\RR^{3},\;\mathfrak{Q_{1}}\ltimes\{0\},\;\mathfrak{Q_{1}}\ltimes\RR,\;\mathfrak{Q_{1}}\ltimes\RR^{3},\;\mathfrak{Q_{2}}\ltimes\{0\},\;\mathfrak{Q_{2}}\ltimes\RR^{2},\;\mathfrak{Q_{2}}\ltimes\RR^{3},\;\mathfrak{N}\ltimes\{0\},\;\mathfrak{N}\ltimes\RR^{3},\;\mathfrak{N_{1}}\ltimes\RR,\;\mathfrak{N_{2}}\ltimes\RR^{2},\;\sL\ltimes\{0\},\;\sL\ltimes\RR^{3},\;\mathfrak{P_{\beta}}$ for $\beta\in \RR\setminus\{0\}$. By taking the Lie subgroups corresponding to these Lie subalgebras, part 1 of the Lemma follows.\\
	
	Similarly for part(2), one can show that the only unimodular subalgebras of $\slr$ containing $\mathfrak{H_{\alpha}}$ for $\alpha\neq 0$ which is the Lie algebra of $H_{\alpha}$ are $\mathfrak{V_{1}}\ltimes\RR^{2},\;\mathfrak{V_{1}}\ltimes\RR^{3},\;\mathfrak{V}\ltimes\RR^{2},\;\mathfrak{V}\ltimes\RR^{3},\;\mathfrak{W}\ltimes\RR^{2},\;\mathfrak{W}\ltimes\RR^{3},\;\mathfrak{K_{t}}\ltimes\RR^{3},\;\mathfrak{Q_{1}}\ltimes\RR^{3},\;\mathfrak{Q_{2}}\ltimes\RR^{2},\;\mathfrak{Q_{2}}\ltimes\RR^{3},\;\mathfrak{N}\ltimes\RR^{3},\;\mathfrak{N_{2}}\ltimes\RR^{2},\;\sL\ltimes\RR^{3},\;\mathfrak{A_{\alpha}},\;\mathfrak{B_{\alpha}},\;\mathfrak{P_{\beta}}$ for $\beta\in \RR\setminus\{0\}$. By the correspondence between Lie groups and Lie algebras, the conclusion of the Lemma holds.
\end{proof}

We are now ready for the proof of Theorem \ref{thm:main}.

\begin{proof}
	By lemma $\ref{l5}$, there exists $\lambda, \mu\in \RR\setminus\{0\}$ and $(g,v)\in \SL(3,\RR)\ltimes \RR^{3}$ such that $\lambda Q_{\xi}((g,v).x)=(Q_{0})_{(0, 0,\alpha)}(x)$ and $\mu L((g,v).x)=L_{0}(x)$. Then by Lemma \ref{l}, it is straight forward to check that $$\SO((Q_{0})_{(0, 0,\alpha)},L_{0})=\left( \begin{pmatrix}
  1&t&\frac{t^{2}}{2}\\
  0&1&t\\
  0&0&1
  \end{pmatrix},
  \begin{pmatrix}
   \frac{\alpha t^{2}}{2}\\
  \alpha t\\
  0
  \end{pmatrix}\right)=H_{\alpha}$$
  and hence $\SO(Q_{\xi},L)=(g,v)H_{\alpha}(g,v)^{-1}$.\\
  Since $H_{\alpha}$ is a unipotent subgroup of $G$, by Ratner's orbit closure theorem \cite{Ratner} there is a closed connected subgroup $F_{\alpha}$ of $G$ such that 
  \begin{enumerate}
  	\item $H_{\alpha}\subset F_{\alpha}$
  	\item the image $[F_{\alpha}.(g,v)^{-1}]$ of $F_{\alpha}.(g,v)^{-1}$ in $G/\Gamma$ is closed and has finite $F_{\alpha}$- invariant measure.
  	\item the closure of $[H_{\alpha}.(g,v)^{-1}]$ is equal to $[F_{\alpha}.(g,v)^{-1}]$ in $G/\Gamma$ .
  \end{enumerate}
Let $x=(g,v)^{-1}\Gamma \in G/\Gamma$. By (2), $F_{\alpha}x$ is closed and has finite $F_{\alpha}$-invariant measure which implies that $F_{\alpha}$ contains a lattice and hence it is unimodular.
Define $f: \RR^{3}\rightarrow \RR^{2}$ by $$f(x)=(Q_{\xi}(x),L(x)).$$
Then 
\begin{align*}
\overline{f(\ZZ^{3})}&= \overline{f((g,v)H_{\alpha}(g,v)^{-1}\Gamma \ZZ^{3})}\\
&\supseteq f\left(\overline{(g,v)H_{\alpha}(g,v)^{-1}\Gamma \ZZ^{3}}\right) \\
&= f\left(\overline{(g,v)F_{\alpha}(g,v)^{-1}\ZZ^{3}}\right). 
\end{align*}
Case 1: Suppose $L(\xi)=0$. Then $\alpha =0$ and by Lemma $\ref{l8}$, $F_{0}$ has to be one of the subgroups $V_{1}\ltimes\{0\},\; V_{1}\ltimes \RR,\; V_{1}\ltimes \RR^{2}, V_{1}\ltimes \RR^{3},\; V\ltimes\{0\},\; V\ltimes \RR,\; V\ltimes \RR^{2}, V\ltimes \RR^{3},\;W\ltimes\{0\},\; W\ltimes \RR,\; W\ltimes \RR^{2}, W\ltimes \RR^{3},\;v(t)\SO(Q_{0})^{\circ}v(t)^{-1}\ltimes \{0\},\; v(t)\SO(Q_{0})^{\circ}v(t)^{-1}\ltimes \RR^{3},\; Q_{1}\ltimes\{0\},\; Q_{1}\ltimes \RR,\; Q_{1}\ltimes\RR^{3},\;Q_{2}\ltimes\{0\},\; Q_{2}\ltimes \RR^{2},\; Q_{2}\ltimes\RR^{3},\;N^{\circ}\ltimes \{0\},\;N^{\circ}\ltimes \RR^{3},\; N_{1}^{\circ}\ltimes \RR,\; N_{2}^{\circ}\ltimes \RR^{2},\;\SL(3,\RR)\ltimes\{0\},\; \SL(3,\RR)\ltimes\RR^{3},\; P_{\beta}$ for $\beta \in \RR\setminus\{0\}$.\\
If $F_{0}$ is one of the subgroups $V_{1}\ltimes\RR^{3},\; V\ltimes \RR^{3},\;W\ltimes \RR^{3},\;Q_{1}\ltimes \{0\},\; Q_{1}\ltimes \RR,\; Q_{1}\ltimes\RR^{3},\;v(t)\SO(Q_{0})^{\circ}v(t)^{-1}\ltimes \RR^{3},\;N^{\circ}\ltimes \RR^{3},\;\SL(3,\RR)\ltimes\{0\},\; \SL(3,\RR)\ltimes\RR^{3},\;P_{\beta}$ then it can be easily verified that $\overline{(g,v)F_{0}(g,v)^{-1}\ZZ^{3}}=\RR^{3}$. Therefore, $f(\RR^{3})\subseteq \overline{f(\ZZ^{3})}$. Since $\overline{f(\RR^{3})}=\RR^{2}$ by Lemma $\ref{l5}$, the conclusion of the Theorem holds.\\

Let $P: \SL(3,\RR)\ltimes \RR^{3}\rightarrow \SL(3,\RR)$ denote the natural projection. Since $F_{0}x$ is closed and has finite $F_{0}$-invariant measure, $(g,v)F_{0}(g,v)^{-1}\cap \Gamma$ is a lattice in $(g,v)F_{0}(g,v)^{-1}$. Assume that $F_{0}$ is generated by unipotent elements. Then by Borel density theorem(4.7.1, \cite{WM}) $(g,v)F_{0}(g,v)^{-1}\cap \Gamma$ is Zariski dense in $(g,v)F_{0}(g,v)^{-1}$. By Lemma $\ref{l3}$, the Zariski closure of $(g,v)F_{0}(g,v)^{-1}\cap \Gamma$ is defined over $\QQ$ and hence $(g,v)F_{0}(g,v)^{-1}$ is defined over $\QQ$. Therefore $P((g,v)F_{0}(g,v)^{-1})= gP(F_{0})g^{-1}$ is also defined over $\QQ$. Hence its normalizer $N(gP(F_{0})g^{-1})=gN(P(F_{0}))g^{-1}$ is defined over $\QQ$.\\

Suppose $F_{0}=v(t)\SO(Q_{0})^{\circ}v(t)^{-1}\ltimes \{0\}$. Since $F_{0}$ is a conjugate of $H$ and $H$ is generated by unipotent elements (by Lemma \ref{l2}) by the above argument we get that $(g,v)F_{0}(g,v)^{-1}$ is defined over $\QQ$. It can be checked that $$\SO(Q'_{\xi})^{\circ}= \SO(Q_{\xi}-2tL^{2})^{\circ}=(g,v)F_{0}(g,v)^{-1}$$ where $Q'=Q-2tL^{2}$. Hence by Lemma \ref{l4}, $Q'_{\xi}$ is not an irrational quadratic form which implies that $Q_{\xi}-2tL^{2}$ is not an irrational quadratic form which is a contradiction.\\
Suppose $F_{0}$ is such that $P(F_{0})$ is either $W$ or $Q_{2}$. Since $N(P(F_{0}))$ is a parabolic subgroup defined over $\QQ$ and $gN(P(F_{0}))g^{-1}$ is also a parabolic subgroup defined over $\QQ$ by (Theorem 20.9, \cite{Borel-book}), there exists $\theta \in \SL(3,\QQ)$ such that $$\theta g N(P(F_{0}))g^{-1}\theta^{-1}=N(P(F_{0})).$$ Therefore $\theta g$ normalises $N(P(F_{0}))$ which implies that $\theta g \in N(P(F_{0}))$ since the normalizer of a parabolic subgroup is the subgroup itself (Theorem 11.16, \cite{Borel-book}). Let $\theta g=h$ where $h\in N(P(F_{0}))$. Then
\begin{align*}
L^{2}(x)&= L_{0}^{2}((g,v)^{-1}.x)\\
&= L_{0}^{2}(g^{-1}x-g^{-1}v)\\
&=L_{0}^{2}(h^{-1}\theta x-g^{-1}v)\\
&=(\beta(q_{1}x_{1}+q_{2}x_{2}+q_{3}x_{3})+c)^{2}, 
\end{align*}
\noindent for some $\beta, c\in \RR$ and $q_{1},q_{2},q_{3}\in \QQ$. Hence $L^{2}$ is not an irrational quadratic form.\\
Now, let $F_{0}$ be such that $P(F_{0})=V_{1}$. Since $N(V_{1})=DV$, $gDVg^{-1}$ is defined over $\QQ$ and hence its unipotent radical $gVg^{-1}$ is also defined over $\QQ$ (0.23, [5]). Again, since $N(V)=DW$, we get that $gDWg^{-1}$ is defined over $\QQ$ and hence its unipotent radical $gWg^{-1}$ is defined over $\QQ$. Similarly, when $P(F_{0})=V$, it follows that $gWg^{-1}$ is defined over $\QQ$. By the argument as before, this gives that $L^{2}$ is not an irrational quadratic form.\\
If $F_{0}=N^{\circ}\ltimes\{0\}$, then since $F_{0}x$ is closed, $F_{0}\cap (g,v)^{-1}\Gamma (g,v)$ is a lattice in $F_{0}$. Since $W\ltimes\{0\}$ is the unipotent radical of $F_{0}$ and $F_{0}$ is solvable, by (Corollary 8.25, \cite{Rag}) we get that $W\ltimes\{0\}\cap (g,v)^{-1}\Gamma (g,v)$ is a lattice in $W\ltimes \{0\}$. This implies that $(W\ltimes\{0\}) x$ is closed. Similarly if $F_{0}=N_{1}^{\circ}\ltimes \RR$ we get that $(W\ltimes \RR) x$ is closed and if $F_{0}=N_{2}^{\circ}\ltimes \RR^{2}$ then $(W\ltimes \RR^{2})x$ is closed. In each of these cases using the same argument as when $P(F_{0})=W$, one can show that $L^{2}$ is not an irrational quadratic form.\\

Case 2: Suppose $L(\xi)\neq 0$. Then $\alpha \neq 0$ and by Lemma $\ref{l8}$, $F_{\alpha}$ is one of the subgroups  $V_{1}\ltimes \RR^{2}, V_{1}\ltimes \RR^{3},\; V\ltimes \RR^{2}, V\ltimes \RR^{3},\;W\ltimes \RR^{2}, W\ltimes \RR^{3},\; v(t)\SO(Q_{0})^{\circ}v(t)^{-1}\ltimes \RR^{3},\; Q_{1}\ltimes\RR^{3},\; Q_{2}\ltimes \RR^{2},\; Q_{2}\ltimes\RR^{3},\;N^{\circ}\ltimes \RR^{3},\; N_{2}^{\circ}\ltimes \RR^{2},\; \SL(3,\RR)\ltimes\RR^{3},\;A_{\alpha},\;B_{\alpha},\; P_{\beta}$ for $\beta \in \RR\setminus\{0\}$.\\
If $F_{\alpha}$ is one of the subgroups $V_{1}\ltimes \RR^{3},\; V\ltimes \RR^{3},\; W\ltimes \RR^{3},\; v(t)SO(Q_{0})^{\circ}v(t)^{-1}\ltimes \RR^{3},\; Q_{1}\ltimes\RR^{3},\;Q_{2}\ltimes\RR^{3},\;N^{\circ}\ltimes \RR^{3},\; \SL(3,\RR)\ltimes \RR^{3}, P_{\beta} $ for $\beta \in \RR\setminus\{0\}$
then $\overline{(g,v)F_{\alpha}(g,v)^{-1}\ZZ^{3}}=\RR^{3}$. Hence $\overline{f(\ZZ^{3})}\supseteq f(\RR^{3})$ which implies $\overline{f(\ZZ^{3})}=\RR^{2},$ since by Lemma \ref{l5}, $\overline{f(\RR^{3})}=\RR^{2}$ .\\
 If $F_{\alpha}$ is such that $P(F_{\alpha})=V_{1},V,W$ or $Q_{2}$, then by the same argument as in case 1, we get that $L^{2}$ is not an irrational quadratic form.\\
 If $F_{\alpha}=N_{2}^{\circ}\ltimes \RR^{2}$, then again by the same argument as in case 1, we get that $L^{2}$ is not an irrational quadratic form.

		\end{proof}


\begin{thebibliography}{99}
	
\bibitem{BM} M.Bachir Bekka and Matthias Mayer, \textit{Ergodic theory and topological dynamics of group actions on homogeneous spaces}, London Mathematical Society Lecture Note Series 269.
\bibitem{Borel-book} A. Borel, \textit{Linear algebraic groups}, Graduate Texts in Mathematics
Series Volume 126, 1991 Springer-Verlag New York.
\bibitem{Borel95} A. Borel, \textit{Values of indefinite quadratic forms at integer point and flows on spaces of lattices}, Bull. Amer. Math. Soc. 32 (1995), 184--204.
\bibitem{Bo16} J. Bourgain, A quantitative Oppenheim theorem for generic diagonal quadratic forms, Israel J. Math., 215 (2016), 503--512.
\bibitem{Dani00} S. G. Dani, \textit{On values of linear and quadratic forms at integral points}, in: Number Theory, in: Trends Math., Birkh\"{a}user, Basel, 2000, pp. 107--119.
\bibitem{Dani08} S. G. Dani, \textit{Simultaneous Diophantine approximation with quadratic and linear forms}, J. Mod. Dyn. 2 (2008) 129--139 (special issue dedicated to G.A. Margulis).
\bibitem{DM90} S.G. Dani, G.A. Margulis, \textit{Orbit closures of generic unipotent flows on homogeneous spaces of $\SL(3, \mathbb{R})$}, Math. Ann. 286 (1990) 101--128.
\bibitem{DM93} \bysame, \textit{Limit distributions of orbits of unipotent flows and values of quadratic forms} in I. M. Gelfand Seminar, Adv. Soviet Math. 16, Amer. Math. Soc., Providence, 1993, 91--137.
\bibitem{EMM1} A. Eskin, G. A. Margulis, and S. Mozes, \textit{Upper bounds and asymptotics in a quantitative version of the Oppenheim conjecture}, Ann. of Math. (2) 147 (1998), 93--141.
\bibitem{EMM2} \bysame, \textit{Quadratic forms of signature $(2, 2)$ and eigenvalue spacings on rectangular $2$-tori}, Ann.of Math. (2) 161 (2005), 679--725.
\bibitem{GGN} A. Ghosh, A. Gorodnik, A. Nevo, \textit{Optimal density for values of generic polynomial maps}, preprint, (2018) arxiv.1801.01027.
\bibitem{GK} A. Ghosh and D. Kelmer, \textit{A quantitative Oppenheim Theorem for generic ternary quadratic forms}, Journal of Modern Dynamics, Volume 12, 2018, 1--8.
\bibitem{Goro04a} A. Gorodnik, \textit{Oppenheim conjecture for pairs consisting of a quadratic form and a linear form}, Trans. Amer. Math. Soc. 356 (11) (2004) 4447--4463.
\bibitem{Goro04b} A. Gorodnik, \textit{Oppenheim-type conjecture for systems of quadratic forms}, Israel J. Math. 356 (11) (2004) 4447--4463.
\bibitem{Lazar} Youssef Lazar, \textit{Values of pairs involving one quadratic form and one linear form at $S$-integral points}, Journal of Number Theory 181 (2017) 200--217.
\bibitem{Lazar19} Youssef Lazar, \textit{On the density of $S$-adic integers near some projective $G$-varieties}, https://arxiv.org/abs/1904.10609.
\bibitem{LM} E. Lindenstrauss and G. Margulis, \textit{Effective estimates on indefinite ternary forms}, Israel Journal of Mathematics
October 2014, Volume 203, Issue 1, pp 445--499.
\bibitem{Ma89} G. Margulis, \textit{Discrete subgroups and ergodic theory, Number theory, trace formulas and discrete groups}, (Oslo, 1987), 377–398, Academic Press, Boston, MA, 1989.
\bibitem{Margulis-book} G. A. Margulis, \textit{Discrete subgroups of semisimple Lie groups}, Ergebnisse der Mathematik und ihrer Grenzgebiete. 3. Folge Volume 17 Springer-Verlag, 1991.
\bibitem{Ma97} G. A. Margulis, \textit{Oppenheim Conjecture}, Fields Medalists' lectures, pp. 272--327, World Sci. Publishing, River Edge, NJ, 1997.
\bibitem{MaMo} Gregory Margulis and Amir Mohammadi, \textit{Quantitative version of the Oppenheim conjecture for inhomogeneous quadratic forms}, Duke Math. J.
Volume 158, Number 1 (2011), 121--160.
\bibitem{Marklof} J. Marklof, \textit{Pair correlation densities of inhomogeneous quadratic forms}, Ann. of Math. (2) 158 (2003), 419 --471.
\bibitem{Rag} M. S. Raghunathan, \textit{Discrete subgroups of Lie groups}, Ergebnisse der Mathematik und ihrer Grenzgebiete. 2. Folge
Series Volume 68, Springer Verlag, 1972.
\bibitem{Ratner} M. Ratner, \textit{Raghunathan'€™s topological conjecture and distributions of unipotent flows}, Duke Math. J. 63 (1991), 235--280.
\bibitem{Sargent} O. Sargent, \textit{Density of values of linear maps on quadratic surfaces}, Journal of Number Theory
Volume 143, October 2014, 363--384.
\bibitem{pw} Pavel Winternitz, \textit{Subalgebras of Lie algebras. Example of $\sL$}, CRM Proceedings and Lecture Notes Volume 34, 2004, 215--227.
\bibitem{WM} Dave Witte Morris, \textit{Ratner's theorems on unipotent flows}, Chicago Lectures in Mathematics, 2005.
	
\end{thebibliography}
\end{document}